\documentclass{article}
\usepackage[a4paper]{geometry}
\usepackage{graphicx}
\usepackage{amsmath,amssymb,amsfonts}
\usepackage{xcolor}
\usepackage{float}
\usepackage[shortlabels]{enumitem}
\usepackage{algorithm}
\usepackage{algorithmic}
\usepackage[backref]{hyperref}
\usepackage{listings}
\usepackage{multirow}
\usepackage{subcaption}
\usepackage{rotating}

\usepackage{amsthm}

\newtheorem{theorem}{Theorem}

\newtheorem{proposition}{Proposition}

\theoremstyle{definition}
\newtheorem{definition}{Definition}

\theoremstyle{remark}
\newtheorem{remark}{Remark}

\newcommand{\argmin}{\operatorname{arg\,min}}
\renewcommand{\algorithmicrequire}{\textbf{Input:}}

\begin{document}

\title{Latent Structural Categorical Matrix Completion with Application to Quasispecies Analysis
}

\author{
Qian Zhang\thanks{Engineering Systems and Design, Singapore University of Technology and Design, Singapore (\texttt{qian\_zhang@sutd.edu.sg})}
\and
Meixia Lin\thanks{Institute of Statistics and Big Data, Renmin University of China, P.R. China (\texttt{lin\_meixia@ruc.edu.cn})}
}

\date{June 6, 2026}

\maketitle

\begin{abstract}
Matrix completion has been extensively studied for real-valued data, but existing methods are often limited in handling categorical variables.
We propose LCMC, a double-loop optimization framework for  categorical matrix completion via latent factorization based on a binary tensor representation. In this setting, each categorical entry is encoded as a one-hot vector along a third tensor mode, thereby preserving its discrete, non-ordinal nature. The outer loop adaptively estimates the latent dimension by iteratively updating it with feedback from the inner loop, while the inner loop reconstructs the categorical matrix through tensor factorization, supported by a corresponding theoretical analysis.
To further improve scalability and robustness, we introduce enhancements including a split-merge-refine strategy and an adaptive data reduction technique. Experiments on synthetic and real-world datasets in viral quasispecies reconstruction, demonstrate that LCMC achieves superior accuracy and efficiency compared to existing methods. 
\end{abstract}

\section{Introduction}
\label{intro}
Matrix completion is a fundamental task in machine learning and signal processing that involves recovering a low-rank matrix from a subset of observed entries. It arises in a wide range of applications, including recommender systems, image inpainting, system identification, and bioinformatics. The core assumption is that the data matrix can be well approximated by a low-rank structure due to latent factors governing the data generation process. Given noisy and partial observations, a typical model is:
\begin{eqnarray}
	\min_{X \in \mathbb{R}^{n \times m}} &\ \left\{  \mathrm{rank}(X) \ \middle\vert \ |X_{ij} - R_{ij}| \leq \delta, \ (i,j) \in \Omega \right\},\label{eq: rank_matrix}
\end{eqnarray}
where $\Omega\subseteq [n] \times [m]$ denotes the set of observed indices, with the shorthand $[n] = \{1,\cdots,n\}$, $R\in\mathbb{R}^{n\times m}$ is the observed data matrix on $\Omega$, and $\delta \geq 0$ controls the approximation tolerance. When an upper bound of the rank, denoted as $k$, is known or can be estimated, problem \eqref{eq: rank_matrix} is often reformulated as:
\begin{eqnarray}
	\min_{X \in \mathbb{R}^{n \times m}} &\ \left\{  \frac{1}{2} \| P_{\Omega}(X - R) \|_F^2\  \middle \vert \  \mathrm{rank}(X) \leq k\right\},\label{eq: rank_matrix_upper}
\end{eqnarray}
where $P_{\Omega}: \mathbb{R}^{n \times m}\to \mathbb{R}^{n \times m} $ is the projection operator defined by
\begin{eqnarray}\label{eq:def_POmega}
	(P_{\Omega}(X))_{ij} = \left\{\begin{array}{ll} 
		X_{ij}, & \text{ if } (i,j)\in\Omega, \\
		0, &\text{ otherwise},
	\end{array} \right. \quad \forall \ (i,j)\in[n]\times[m].
\end{eqnarray}
Convex relaxations on problem \eqref{eq: rank_matrix_upper} 
have made low-rank matrix completion tractable with theoretical guarantees under mild conditions \cite{candes2012exact}, leading to a variety of algorithms including the singular value thresholding method \cite{cai2010singular}, the alternating minimization method \cite{jain2013low}, and the alternating direction method of multipliers \cite{lin2011linearized}, as well as robustness extensions for handling noise and outliers \cite{candes2011robust}. More recently, nonconvex reformulations of \eqref{eq: rank_matrix_upper} that factorize the matrix as $X=UV$ with $U\in\mathbb{R}^{n\times k}$ and $V\in\mathbb{R}^{k\times m}$ have been explored, giving rise to the unconstrained optimization problem:
\begin{eqnarray*}
	\min_{U\in\mathbb{R}^{n\times k}, V\in\mathbb{R}^{k\times m}} &\quad \frac{1}{2} \| P_{\Omega}(U V - R) \|_F^2 .
\end{eqnarray*}
Although nonconvex, this formulation has been shown to admit no spurious local minima under suitable conditions, allowing efficient recovery via gradient-based methods \cite{ge2016matrix,zheng2016convergence}. Optimization on matrix manifolds further provides a principled and efficient framework for such problem \cite{boumal2019global}.

In many real-world applications, matrix entries are categorical rather than continuous, reflecting inherent structural properties such as user preferences, genotype markers, or survey responses. A well-known example is the Netflix Prize competition, which aims to predict user ratings for movies based on a partially observed user-movie rating matrix, where entries are integers from 1 to 5. Although these ratings are discrete, their ordinal nature allows conventional matrix completion methods, originally developed for continuous data, to be applied by treating the entries as real-valued and using post-processing techniques such as rounding to obtain final predictions. In contrast, some applications involve categorical and nominal data with no natural ordering. One key example arises in computational biology~\cite{posada2017recent}, where the goal is to reconstruct viral quasispecies from partially observed sequencing reads, each corresponding to a fragment of a viral genome composed of nucleotides from the categorical set \{A, C, G, T/U\}. Unlike rating data, these values are nominal, and directly applying real-valued matrix completion methods may obscure important structural properties, leading to suboptimal recovery. 

To address categorical matrix completion with non-ordinal data, one line of research introduces a link function that connects the observed categorical variables to an underlying continuous low-rank factor matrix, with the goal of recovering this latent matrix from the categorical observations.
These methods adopt loss functions suited to categorical variables, such as one-versus-all hinge or multinomial logistic losses, combined with low-rank regularization, and infer categories by decoding predicted scores. For example, generalized low-rank models specify loss functions for each data type, using classification losses for categorical entries while learning continuous low-rank factors \cite{udell2016generalized}. In the binary setting, 1-bit matrix completion models observations as Bernoulli samples whose success probabilities depend on a low-rank latent matrix via the logistic or probit link function, and estimates the matrix by maximizing the likelihood under low-rank constraints \cite{davenport20141}. Cao et al.~\cite{cao2015categorical} extend this framework to the multi-class case, proposing a likelihood-based model with nuclear norm regularization, from which discrete labels are inferred. Whereas, the work \cite{chen2020categorical} adopts a one-versus-rest formulation to estimate the probability distribution over possible categories for each missing entry. In addition, \cite{sun2021evaluation} proposes an alternative strategy based on voting, imputing missing values by aggregating information from observed entries.
Another direction of research tackles the problem by operating on the categorical matrix directly. For instance, in biological data analysis, several studies propose structured factorization approaches for categorical matrix completion with low latent dimension \cite{cai2016structured,hashemi2018sparse,ahn2018viral}. 
Specifically, these approaches encode each categorical entry of the observed data matrix as a one-hot vector in an additional tensor mode, and then decompose the tensor into a left assignment matrix and a right representative tensor. The involved tensor decomposition is often solved by alternating minimization with projected gradient (PG) steps as subsolvers. However, most existing works use only one PG update per iteration; although may be effective in practice, this approach can yield low-accuracy solutions and lacks convergence guarantees.

Despite substantial research efforts, categorical matrix completion remains fundamentally challenging. A major difficulty lies in observational noise, which can distort categorical entries and obscure the latent structural patterns. This issue is especially severe when low-frequency categories, such as minor haplotypes in viral populations, are of primary interest \cite{martinez2022sars,orton2015distinguishing,van2022general}. 
Existing factorization methods \cite{cai2016structured,hashemi2018sparse} often assume balanced category distributions, which may be reasonable for diploid or polyploid organisms but is invalid in more heterogeneous settings. Under such conditions, these methods tend to capture dominant categories effectively but struggle to recover sparse or rare signals, particularly with missing data and binary constraints. Another major challenge is the need to pre-specify the number of latent groups, a quantity that is typically unknown and must be inferred from the data. An incorrect estimate can lead to poor recovery of the underlying categorical structure. Finally, scalability poses a significant barrier in large-scale applications such as sequencing, recommendation, and social network analysis, where categorical matrices are both large and sparse. Solving the completion problem under these setting requires algorithmic innovations that preserve categorical structure while balancing computational efficiency, memory usage, and solution quality.

This work addresses the problem of latent structural  categorical matrix completion. 
Let 
\begin{eqnarray}
	\mathcal{A}^{n\times m}:=\left\{A \ \middle| \ A \text{ is an } n \times m \text{ matrix with entries from } l \text{ categories} \right\}\label{eq: def_A}
\end{eqnarray}
denote the set of $n\times m$ categorical matrices, where each entry is drawn from a finite set of $l$ possible labels. Given an unknown categorical matrix ${R}^*\in \mathcal{A}^{n \times m}$, we observe only a subset of its entries, possibly corrupted by noise:
\begin{eqnarray}\label{eq:def_R}
	R = P_{\Omega}({R}^*+E), 
\end{eqnarray}
where $\Omega \subseteq [n] \times [m]$ denotes a given set of observed indices, and $+E$ models potential label-flipping errors rather than standard numeric perturbations. The operator $P_{\Omega}$ retains observed entries and masks unobserved ones. For consistency, we adopt a unified notation $P_{\Omega}$ for both real-valued and categorical matrices: for real-valued matrices (as defined in \eqref{eq:def_POmega}), unobserved entries are set to zero, whereas in the categorical case, they are assigned a special symbol ``$\perp$". Accordingly, the space of partially observed categorical matrices from $\mathcal{A}^{n\times m}$ with index set $\Omega$ is defined as
\begin{eqnarray}\label{eq: def_AOmega}
	\mathcal{A}^{n \times m}_\Omega := 
	\left\{ 
	A \ \middle| \ 
	A_{ij} \text{ takes a label from } l \text{ categories} \quad\right. \\ \nonumber \left. \text{ if } (i,j) \in \Omega; \ 
	A_{ij} = \perp \text{otherwise} 
	\right\}.
\end{eqnarray}
Our framework differs in that the latent structure is imposed directly on the categorical data itself, rather than through an auxiliary continuous low-rank matrix. In particular, the latent structure corresponds to grouping fibers into latent classes, leading to discrete structural constraints. Moreover, the complexity of the latent dimension is not fixed a priori, but is adaptively determined by the proposed outer loop.

Our goal is to reconstruct $R^*$ from the partial, noisy observations $R$ by simultaneously recovering missing entries and correcting label-flip errors.
Similar to low-rank real-valued matrix completion, this requires estimating the unknown latent dimension $k$, which in the categorical setting represents the number of latent groups. 
To address this, we propose a double-loop approach to recover the categorical matrix ${R}^*$ with low latent dimension. 
In the outer loop, the number of latent groups $k$ is adaptively updated by a tailored algorithm that leverages intermediate results from the inner loop. 
The inner loop then factorizes ${R}$ for the given $k$ into an assignment matrix $U \in \{0,1\}^{n \times k}$, whose one-hot rows indicate latent group memberships, and a representative matrix $V \in \mathcal{A}^{k \times m}$, where each row specifies the categorical pattern of a latent group.
The proposed approach is as follows.

\vspace{0.5em}

\fbox{
	\parbox{0.9\linewidth}{
		\textbf{LCMC: Double-loop latent structural categorical matrix completion} \\
		\textbf{Outer loop (latent dimension estimation):} Iteratively update the estimated number of latent groups $k$ using feedback from the inner loop.\\
		\textbf{Inner loop (factorization):} Given latent dimension \(k\), factorize \(R\) into \(U \in \{0,1\}^{n\times k}\) and \(V \in \mathcal{A}^{k \times m}\) by solving
		\begin{eqnarray}\label{pr:categoricalMF}
			\min_{U\in \{0,1\}^{n\times k},\, V\in \mathcal{A}^{k \times m}} \mathcal{D}_{\Omega} \left(P_{\Omega}(UV),\ P_{\Omega}(R)\right).
		\end{eqnarray}
		\vspace{-0.3cm}
	}
} \\  
\vspace{0.3em}

\noindent Here, $\mathcal{D}_{\Omega}(\cdot,\cdot)$ is the restricted discrepancy measure for comparing categorical data on the observed index set $\Omega$. For any $M_1, M_2\in\mathcal{A}_{\Omega}^{n\times m}$, it is defined as
\begin{eqnarray}\label{eq:def_categorical_distance}
	\mathcal{D}_{\Omega}(M_1, M_2) := \sum_{(i,j)\in\Omega} \mathbf{1}_{\{(M_1)_{ij}\neq (M_2)_{ij}\}},
\end{eqnarray}
where the inequality $(M_1)_{ij}\neq (M_2)_{ij}$ denotes a mismatch between categorical labels, and $\mathbf{1}_{\{\cdot\}}$ is the indicator function, defined as $\mathbf{1}_{\{C\}}=1$ if the condition $C$ is true, and $0$ otherwise.

The main contributions of this work are summarized as follows:
\begin{itemize}[left=0.5em]
	\item We propose LCMC, a double-loop optimization approach for \underline{L}atent structural \underline{C}ategorical \underline{M}atrix \underline{C}ompletion. In the outer loop, we estimate the intrinsic dimension (i.e., the number of latent groups) using feedback from the inner loop, based on a statistical criterion. In the inner loop, we reformulate the categorical completion task with given latent dimension into a tensor representation and solve the resulting factorization problem, together with a convergence analysis.
    
	\item To improve scalability and robustness, we incorporate two algorithmic enhancements into LCMC, including a split-merge-refine strategy for handling large-scale, high-dimensional datasets, and an adaptive data reduction technique to mitigate issues arising from category imbalance and sparse observations.
	
	\item We evaluate our method on the viral quasispecies reconstruction problem, using both synthetic and real-world datasets. Experimental results show that LCMC outperforms the state-of-the-art method in reconstruction accuracy and computational efficiency, particularly in challenging settings with noise, category imbalance, and large-scale data.
\end{itemize}

The remainder of this paper is structured as follows. In Section \ref{sec: formulation}, we formulate the categorical matrix completion problem with a tensor-based representation. Section \ref{sec: BCMC} introduces our double-loop framework, LCMC, and presents its convergence analysis. In Section \ref{sect:enhanced}, we introduce enhanced techniques designed to improve scalability and robustness. Section \ref{sec: viral} presents numerical experiments on an application to the viral quasispecies reconstruction. Finally, Section \ref{sec: conclusion} concludes the paper and outlines the potential directions for future work.

\section{Problem formulation}
\label{sec: formulation}	
To solve the categorical matrix completion problem, we embed the categorical domain into a structured numerical form. Specifically, we adopt a tensor-based formulation where categorical values are encoded along an additional mode of a higher-order binary tensor, which preserves the discrete, non-ordinal nature of the data while providing a structured foundation for algorithms designed to operate effectively on categorical structures.

\subsection{Notations}\label{sec: notation}
We denote matrices using capital letters and tensors using calligraphic letters. For a matrix $R\in \mathbb{R}^{I_1\times I_2}$, $R_{ij}$ denotes its $(i,j)$-th entry,  
and $R_{i:}$ denotes its $i$-th row. 
For a tensor $\mathcal{R}\in \mathbb{R}^{I_1\times I_2\times I_3}$, $\mathcal{R}_{ijk}$ denotes its $(i,j,k)$-th entry, 
$\mathcal{R}_{ij:}$ denotes its mode-3 fiber indexed by $(i,j,:)$, and $\mathcal{R}_{i::}$ denotes its $i$-th horizontal slice indexed by $(i,:,:)$. 
We denote $\|x\| = \left(\sum_{i=1}^n x_i^2\right)^{1/2}$ for $x\in\mathbb{R}^n$, $\|X\|=\|X\|_F =  (\sum_{i=1}^{n_1} \sum_{j=1}^{n_2} X_{ij}^2) ^{1/2}$ for $X\in\mathbb{R}^{n_1\times n_2}$, and $\|\mathcal{X}\| = ( \sum_{i=1}^{n_1} \sum_{j=1}^{n_2} \sum_{k=1}^{n_3} \mathcal{X}_{ijk}^2 )^{1/2}$ for $\mathcal{X}\in\mathbb{R}^{n_1\times n_2 \times n_3}$.
Let $\mathbf{0}$ denote the zero vector, with dimension inferred from context.
Define \( e_i^d \) as the row vector with a \( 1 \) in the \( i \)-th position and \( 0 \)s elsewhere;
the superscript \( d \) is omitted when clear from context.
Define $[n] := \{ 1,  \ldots, n \}$. 
For $\Omega \subseteq [n]\times [m]$, let $|\Omega|$ be its cardinality, $\Omega^{c} := \{(i,j)\ |\ (i,j)\notin \Omega,\ i \in [n],\ j\in [m]\}$ be its complement, and $\Omega_i := \{j \in [m] \ |\ (i,j)\in\Omega\}$ for each $i\in[n]$. 

\subsection{Binary tensor representation}
We reformulate problem~\eqref{pr:categoricalMF} by mapping the categorical matrices into binary tensor representations.
Let $M\in\mathcal{A}^{n\times m}$, where $\mathcal{A}^{n\times m}$ is the set of categorical matrices defined in  \eqref{eq: def_A}. Without loss of generality, we index the $l$ categories involved in $\mathcal{A}^{n\times m}$ by $\{1,2,\ldots,l\}$. The categorical information in $M$ can be equivalently represented by a binary tensor $\mathcal{M}\in\{0,1\}^{n\times m\times l}$ as follows.  
For each entry $M_{ij}$, the corresponding fiber $\mathcal{M}_{ij:}\in\{0,1\}^{l}$ is a one-hot vector encoding the category of $M_{ij}$, that is,
\begin{eqnarray*}
	\mathcal{M}_{ijq} = \left\{
	\begin{aligned}
		1, &\quad \mbox{if } M_{ij} \mbox{ belongs to category } q, \\
		0, & \quad \mbox{otherwise},
	\end{aligned}\right.
	\quad \forall i\in[n], \, j\in[m], \, q\in[l].
\end{eqnarray*}
For $M\in \mathcal{A}_{\Omega}^{n\times m}$, where $\mathcal{A}_{\Omega}^{n\times m}$ is the set of partially observed categorical matrices defined in \eqref{eq: def_AOmega}, we define its tensor representation $\mathcal{M}\in\{0,1\}^{n\times m\times l}$ by setting
$\mathcal{M}_{ij:}$ as above for $(i,j) \in \Omega$, and assigning $\mathcal{M}_{ijq}=0$ for all $q\in[l]$ whenever $(i,j)\notin\Omega$. For $M_1, M_2\in\mathcal{A}_{\Omega}^{n\times m}$, with corresponding tensor respresentations $\mathcal{M}_1,\ \mathcal{M}_2\in\{0,1\}^{n\times m\times l}$, the restricted distance defined in \eqref{eq:def_categorical_distance} can be equivalently computed in tensor form as
$$
\mathcal{D}_\Omega(M_1, M_2) = \frac{1}{2} \sum_{(i,j)\in\Omega} \|(\mathcal{M}_1)_{ij:} - (\mathcal{M}_2)_{ij:}\|^2 = \frac{1}{2} \|P_{\Omega}(\mathcal{M}_1 - \mathcal{M}_2)\|^2,
$$ 
where, in line with the earlier definitions of $P_\Omega$ for matrices, we extend the operator to third-order tensors in a consistent manner: for a tensor $\mathcal{M}\in \mathbb{R}^{n\times m\times l}$, $P_\Omega(\mathcal{M})$ acts on each $(i,j)$-indexed fiber $\mathcal{M}_{ij:}$ by retaining it if $(i,j) \in \Omega$, and setting it to zero otherwise.

Given the number of latent groups $k$, let $\mathcal{R}\in \{0,1\}^{n\times m\times l}$ and $\mathcal{V}\in \{0,1\}^{k\times m\times l}$ denote the tensor representations of the categorical matrices $R\in\mathcal{A}^{n\times m}_{\Omega}$ and $V\in\mathcal{A}^{k\times m}$. Problem~\eqref{pr:categoricalMF} is equivalently reformulated as  
\begin{align}
	\min_{U \in \mathcal{F}_U,\ \mathcal{V} \in \mathcal{F}_{\mathcal{V}}} &\quad \left\{ f(U, \mathcal{V}):=\frac 1 2 \|P_{\Omega} (\mathcal{R} - U \mathcal{V})\|^2\right\} .  \label{pr:factorization} \tag{P0} 
\end{align}
Here, $\mathcal{F}_\mathcal{V}$ denotes the set of representative tensors, where each slice $\mathcal{V}_{p::} \in \{0,1\}^{m \times l}$ encodes the categorical pattern of group $p$: 
\begin{eqnarray*}
	\mathcal{F}_\mathcal{V}:=\ \left\{ \mathcal{V}\in \{0,1\}^{k\times m \times l}\ \middle \vert\ \sum\nolimits_{q=1}^l \mathcal{V}_{p jq} = 1, \ p \in [k], \ j\in [m] \right\}, 
\end{eqnarray*}
and $\mathcal{F}_U$ denotes the set of assignment matrices with one-hot rows indicating latent group membership:
$$\mathcal{F}_U:=\ \left\{ \left.U \in \{0,1\}^{n\times k}\ \right|\  U_{i:}\in\{e_1, e_2, \ldots, e_k\}, \ i\in [n]  \right\}.$$
	
	\section{LCMC: Double-loop \underline{L}atent Structural \underline{C}ategorical \underline{M}atrix \underline{C}ompletion}
	\label{sec: BCMC}
	In this section, we present our double-loop approach for latent structural categorical matrix completion. 
	In the outer loop, we employ a binary search approach using feedback from the inner loop, guided by a variance-based criterion, to iteratively learn the latent dimension $k$. In the inner loop, given a fixed $k$, we reconstruct the categorical matrix by a tailored tensor factorization method for solving \eqref{pr:factorization}. 
	    
	\subsection{Outer loop: Binary search for latent dimension $k$}\label{sec: upper_level}
	We quantify the deviation between the observed data and a learned tensor representation using a sum of squared error (SSE) criterion restricted to the observed index set $\Omega$. Specifically, for latent dimension $k$, we define
	$${\rm SSE}_k := \|P_{\Omega} (\mathcal{R} - U(k) \mathcal{V}(k))\|^2,$$ where $(U(k),\mathcal{V}(k))$ denotes a solution to the completion problem~\eqref{pr:factorization} with fixed latent dimension $k$. As the latent dimension $k$ increases, ${\rm SSE}_k$ is non-increasing because higher latent dimensional models can better fit the observed data. Therefore, using SSE alone is insufficient to determine the optimal latent dimension, as it may continue to decrease without indicating when overfitting begins.
	To address this, we measure the relative improvement in fit when increasing the latent dimension from $k$ to $k+1$ via the SSE ratio:
	\begin{equation}\label{eq:improvement}
		\rho_k = {\rm SSE}_{k+1}/{\rm SSE}_k.
	\end{equation}
	Figure \ref{fig:upper_level_sse} illustrates how ${\rm SSE}_k$ (blue) and $\rho_k$ (red) evolves as functions of $k$. While SSE steadily decreases with $k$, the ratio $\rho_k$ captures the marginal gain from increasing model complexity. 
	A ratio close to 1 indicates that increasing $k$ results in diminishing improvements. This trend can be used to determine an optimal $k$ by identifying the smallest $k$ for which $\rho_k$ starts to stabilize, suggesting that further increasing $k$ does not significantly improves the performance.
	%\vspace*{-0.5cm}
    \begin{figure}[H]
		\centering	\includegraphics[width=0.45\linewidth]{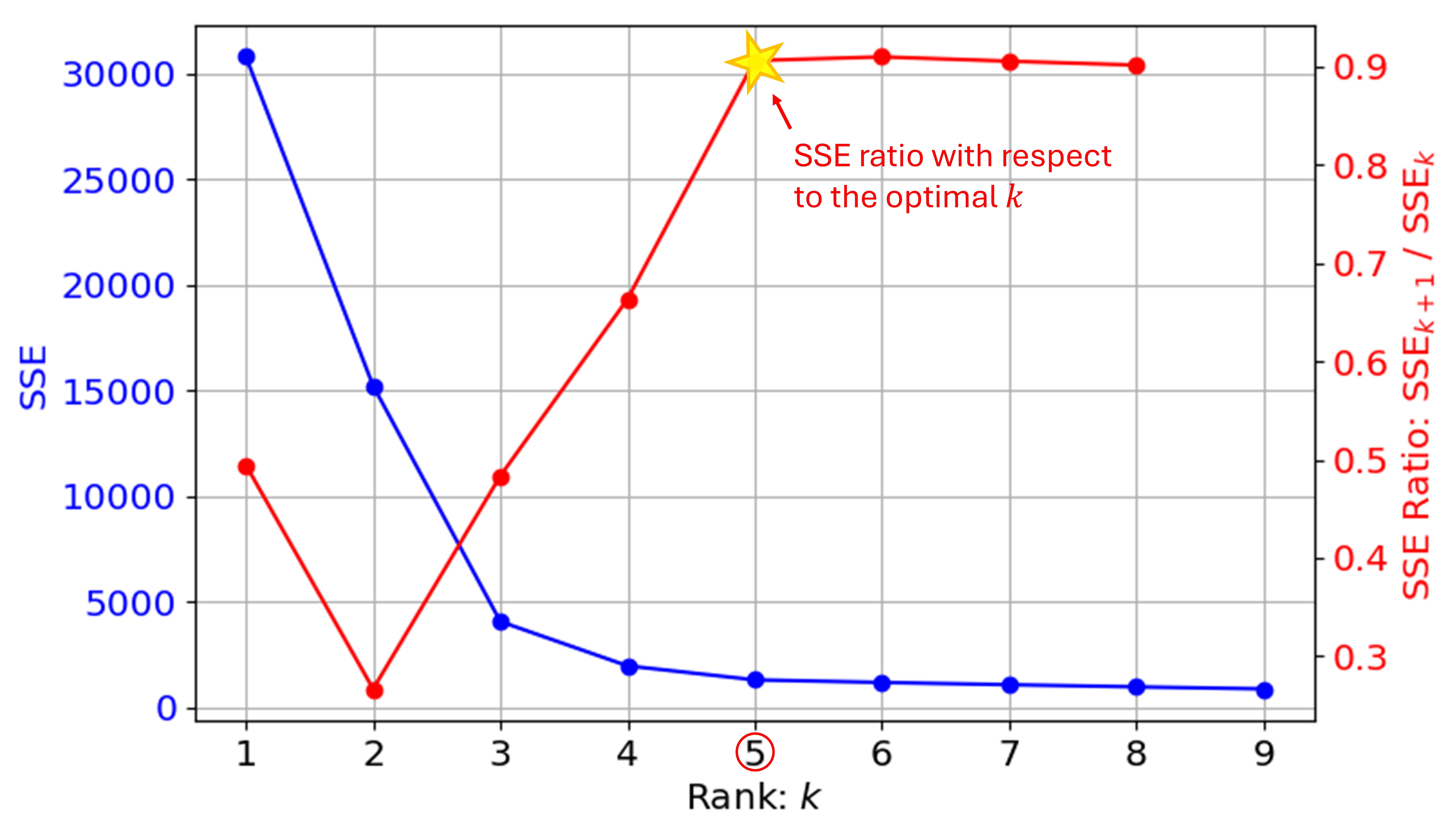}
		\caption{ 
			Illustration of 
			SSE (blue) and its ratio (red) versus $k$ (true latent dimension $\bar{k}\!=\!5$).}
		\label{fig:upper_level_sse}
	\end{figure}
	%\vspace*{-0.5cm}
	To efficiently determine the optimal latent dimension $k$, we apply a binary search strategy based on SSE ratio $\rho_k$ defined in \eqref{eq:improvement}, by iterative halving the search space until the optimal value is identified.
	At each step, $\rho_k$ at the midpoint is evaluated and the interval is updated based on whether $\rho_k > 1 - \epsilon$ for a prescribed threshold $\epsilon > 0$ (in our experiments, we set $\epsilon=10^{-2}$). The procedure terminates when the interval converges, as detailed in Algorithm~\ref{alg:BCMC}.
	%\vspace*{-0.5cm}
	\begin{algorithm}[!h]
		\caption{(Outer loop) Binary search for latent dimension $k$: $( U(k),\mathcal{V}(k)) \gets$ \textbf{LCMC}($\mathcal{R},\Omega$) }
		\label{alg:BCMC}
		\begin{algorithmic}[1]
			\REQUIRE  Data tensor $\mathcal{R}\in \{0,1\}^{n\times m\times l}$, observed index set $\Omega\subseteq [n]\times [m]$. 
			\renewcommand{\algorithmicrequire}{\textbf{Initialization:}} 
			\REQUIRE Set tolerance $\epsilon>0$, $[k_{\min}, k_{\max}]=[1, n]$, Flag $= 0$. Initialize some $k\in[1, n]$.
			\WHILE{$k_{\max} - k_{\min} \neq 1$}
			% \FOR{$k=K,K+1$}
			\STATE Apply \textbf{TRC} (Algorithm \ref{alg:alt_min}) to compute: 
			\vspace{-5pt}
			\[
			( U(k),\mathcal{V}(k) ) \gets \textbf{TRC} (\mathcal{R}, \Omega, k),\quad (U(k+1),\mathcal{V}(k+1)) \gets \textbf{TRC} (\mathcal{R}, \Omega, k+1),
			\]
			\vspace{-18pt}
			\STATE Compute the improvement $\rho_k$ by \eqref{eq:improvement}.
			\IF{$\rho_k > 1 - \epsilon$} 
			\STATE Flag $= 1$. Update $k_{\max} = k$, and $k = \lfloor (k_{\min} + k_{\max})/2 \rfloor$. 
			\ELSE
			\STATE \textbf{if} Flag \textbf{then} $k_{\min} \gets k$, $k \gets \lfloor (k_{\min} + k_{\max})/2 \rfloor$ \textbf{else} $k \gets 2k$ \textbf{end if}
			\ENDIF
			\ENDWHILE
			\ENSURE Representative tensor $\mathcal{V}(k)$ and assignment matrix $U(k)$.  
		\end{algorithmic}  
	\end{algorithm}	
	%\vspace{-0.5cm}

    \begin{remark}
    The proposed outer loop algorithm estimates the latent dimensionality $k$ using a strategy inspired by the elbow method (see, e.g., \cite{tibshirani2001estimating}). Rather than performing a standard exhaustive search over all candidates $k \in \{1,\dots,k_{\max}\}$, it incorporates a binary search scheme, which reduces the number of required evaluations to $O(\log k_{\max})$ and thus substantially lowers the computational cost. 
    However, as in other elbow-based approaches, the selection may become less distinguishable when the SSE ratio decays smoothly without a pronounced elbow, particularly in high-noise settings. In such cases, more principled information-criterion-based \cite{hung2022generalized} or discrepancy-based \cite{li2026robust} model selection methods may be adopted at the cost of increased computation.
    \end{remark}
      
	%------------------------------------------------------------------------------------
	\subsection{Inner loop: Latent structural tensor completion}
	In this subsection, we address the tensor completion problem \eqref{pr:factorization} with a prescribed latent dimension $k$. 
	Note that since both $U$ and $\mathcal{V}$ are subject to discrete constraints, problem \eqref{pr:factorization} falls into the class of integer programming problems, which are generally intractable to solve globally.
	To address this challenge, a common strategy is to relax the discrete constraints to their convex hull \cite{ahn2018viral,hashemi2018sparse}, thereby yielding a convex program whose solution provides a lower bound to the original problem. 
	
	Motivated by this idea, we adopt a partial relaxation framework for solving \eqref{pr:factorization}.
	Specifically, we relax the discrete constraint on $\mathcal{V}$ by replacing the binary domain $\{0,1\}^{k\times m\times l}$ with its convex hull, allowing $\mathcal{V}\in [0,1]^{k\times m\times l}$.
	This relaxation enables each mode-3 fiber of $\mathcal{V}$  to represent a probability distribution over the $l$ possible categories, rather than a single hard assignment. This probabilistic formulation captures uncertainty and provides a more flexible representation of the underlying categorical structure. As a result, $\mathcal{V}$ is now constrained to lie within the convex hull of the original feasible set $\mathcal{F}_{\mathcal{V}}$, defined as 
	\begin{align}
		\bar{\mathcal{F}}_\mathcal{V}:= \left\{ \mathcal{V}\in [0,1]^{k\times m \times l}\ \middle\vert\ \sum_{q=1}^t \mathcal{V}_{p jq} = 1, \ p \in [k], \ j\in [m] \right\}.\nonumber 
	\end{align}
	We retain the constraint $U \in \mathcal{F}_U$, which enforces that each data point belongs to a single latent group while preserving the convergence guarantees of our proposed algorithm.
	The resulting partially relaxed problem of \eqref{pr:factorization} is formulated as:
	\begin{align}
		\min_{U\in \mathcal{F}_U,\; \mathcal{V} \in \bar{\mathcal{F}}_\mathcal{V}} \;  f(U, \mathcal{V}). 
		\label{pr:mec2} \tag{P1} 
	\end{align}
	After obtaining a solution $(U,\mathcal{V})$ to \eqref{pr:mec2}, we obtain a feasible solution to the original problem \eqref{pr:factorization} by projecting the relaxed tensor $\mathcal{V}$ onto the discrete feasible set $\mathcal{F}_{\mathcal{V}}$, while keeping $U$ unchanged. Specifically, the projection of a tensor $\mathcal{V}\in\mathbb{R}^{k\times m\times l}$ onto the set $\mathcal{F}_\mathcal{V}$, denoted by $\operatorname{Proj}_{\mathcal{F}_\mathcal{V}}(\mathcal{V})$, can be computed as:
	\begin{equation*}
		\left(\operatorname{Proj}_{\mathcal{F}_\mathcal{V}}(\mathcal{V})\right)_{pj:}=e_{q^*},\; \textrm{where}\ q^* = \underset{q\in[l]}{\arg\max}\ \mathcal{V}_{pjq},
		\quad \text{for all } p \in [k], \ j\in [m].
	\end{equation*}
	We now introduce an efficient iterative algorithm to solve the relaxed problem~\eqref{pr:mec2}, and thereby addressing the original tensor completion task~\eqref{pr:factorization}. The algorithm alternates structured updates of $U$ and $\mathcal{V}$ while respecting their respective constraints, see Algorithm~\ref{alg:alt_min}.
	
	\begin{algorithm}[H]
		\caption{(Inner loop) Latent structural tensor completion:  $(U,\mathcal{V}) \gets$ \textbf{TRC}($\mathcal{R}, \Omega, k$)}\label{alg:alt_min}
		\begin{algorithmic}[1]
			\REQUIRE Data tensor $\mathcal{R}$, observed index set $\Omega$ and latent dimension $k$. Set $t=0$. Choose tolerance $\eta>0$ (in our experiments, we set $\eta=10^{-5}$). 
			\STATE Compute the initialization:
			\begin{eqnarray}
				\mathcal{V}^0 = \operatorname{Proj}_{\mathcal{F}_\mathcal{V}} (\operatorname{SVD}({P}_{\Omega}(\mathcal{R}), k)) \quad \text{and} \quad
				U^{0} {=}\ \underset{U\in{\mathcal{F}}_{U} }{\argmin} \left\{ %\left.  
				\frac{1}{2} \left\|{P}_{\Omega} \left(\mathcal{R} - {U} \mathcal{V}^{0} \right) \right\|^2  
				\right\}. 
			\end{eqnarray}
			where $\operatorname{SVD}(\mathcal{V}, k)$ denotes the tensor obtained by reshaping the right singular vector matrix from the rank-$k$ truncated SVD of matricization of tensor $\mathcal{V}$.
			\REPEAT
			\STATE 
			Update the representative tensor $\mathcal{V}^{t}$ by 
			\begin{eqnarray}
				\mathcal{V}^{t} =\ \underset{\mathcal{V}\in\mathbb{R}^{k\times m\times l}}{\argmin} 
				\left\{	\hat{f}_t (\mathcal{V}):= \frac 1 2 \left\|\hat{\mathcal{R}}^t - {U}^{t-1}\mathcal{V} \right\|^2 \right\}, \label{pr:ALS-V}
			\end{eqnarray}
			where \begin{eqnarray}\label{eq:imputingR}
				\hat{\mathcal{R}}^{t} := {P}_{\Omega} \left(\mathcal{R} \right) + P_{\Omega^c}   
				\left( {U}^{t-1} \mathcal{V}^{t-1} \right) .
			\end{eqnarray}
			\STATE 
			Compute the assignment matrix ${U}^{t}$ by 
			\begin{eqnarray}
				U^{t} {=}\ \underset{U\in{\mathcal{F}}_{U} }{\argmin} \left\{ %\left.  
				\frac{1}{2} \left\|{P}_{\Omega} \left(\mathcal{R} - {U} \mathcal{V}^{t} \right) \right\|^2 \right\}. \label{pr:ALS-sub-U}
			\end{eqnarray}
			\STATE $t\gets t+1$.
			\UNTIL{%convergence.
				$$\min\left\{ \left\| \mathcal{V}^t - \mathcal{V}^{t-1}\right\| , f(U^t, \mathcal{V}^t), \frac{\left|f(U^t, \mathcal{V}^t) - f(U^{t-1}, \mathcal{V}^{t-1}) \right|}{\max\{1, f(U^{t-1}, \mathcal{V}^{t-1})\}} \right\} \leq \eta.$$
			}
			\STATE Project the tensor $\mathcal{V}^t$ onto the set $\mathcal{F}_\mathcal{V}$: $\mathcal{V}^t = \operatorname{Proj}_{\mathcal{F}_\mathcal{V}}(\mathcal{V}^t)$. 
			\ENSURE Representative tensor $\mathcal{V} = \mathcal{V}^{t}$ and assignment matrix $U=U^t$.
		\end{algorithmic}  
	\end{algorithm}
	
	An appealing property of Algorithm~\ref{alg:alt_min} is that each subproblem admits a closed-form solution, allowing for efficient computation. In particular, each update can be carried out in $\mathcal{O}(nmlk)$ time. These closed-form solutions are summarized in the following proposition.
	
	\begin{proposition}\label{prop:sol_V}
		For each iteration $t$, both subproblems \eqref{pr:ALS-V} and \eqref{pr:ALS-sub-U} in Algorithm \ref{alg:alt_min} 
		admit closed-form solutions as follows.
		\begin{itemize}[leftmargin=2em]
			\item[(i)] Define the index set $J_p:= \left\{i\in[n]\left| ({U}^{t-1})_{i:} = e_p \right. \right\}$ for each $p\in[k]$. A valid solution $\mathcal{V}^t$ to problem \eqref{pr:ALS-V} is given by
			\begin{equation}\label{eq:updateV_avg}
				({\mathcal{V}}^t)_{p::}  = \left\{\begin{array}{ll}
					\frac{1}{|J_p|}\sum_{i\in J_p} (\hat{\mathcal{R}}^t)_{i::} & \text{if }\ J_p\neq \emptyset, \\[4pt]
					({\mathcal{V}}^{t-1})_{p::} & \text{if }\ J_p =  \emptyset,
				\end{array}\right.  \quad \text{for } p\in[k].
			\end{equation}
			Moreover, the above tensor $\mathcal{V}^t$ satisfies $\mathcal{V}^t\in \bar{\mathcal{F}}_{\mathcal{V}}$. 
			\item[(ii)] The solution $U^t$ to problem~\eqref{pr:ALS-sub-U} is:
			\begin{equation}\label{eq:def_assignment_U}
				(U^t)_{i:}=e_{p^*} \quad \textrm{with} \quad p^* \in  \underset{p\in[k]}{\arg\min}\ \|
				P_{\Omega_i}(\mathcal{R}_{i::}-\mathcal{V}^t_{p::})\|^2\ , \quad \textrm{ for } \; i\in [n],  
			\end{equation}
			where $\Omega_i=\{j\in[m]\ |\ (i,j)\!\in\!\Omega\}$. If multiple minimizers exist, the smallest one is selected. 
		\end{itemize}
	\end{proposition}
	\begin{proof}
		Note that the $\mathcal{V}$-subproblem \eqref{pr:ALS-V} is separable across horizontal slices. Thus, the minimization problem in \eqref{pr:ALS-V} can be equivalently written as
		\begin{equation}
			\min_{\mathcal{V}\in \mathbb{R}^{k\times m \times l}}  
			\left\{  
			\frac 1 2 \sum_{p=1}^k   \sum_{i\in J_p} \left\|(\hat{\mathcal{R}}^t)_{i::} - \mathcal{V}_{p::} \right\|^2   \right\}.\label{eq: new_Vsub}
		\end{equation}
		The problem~\eqref{eq: new_Vsub} is convex, so any point at which the gradient vanishes is a global minimizer.
		It is easy to verify that the gradient vanishes at $\mathcal{V}^t$ defined in~\eqref{eq:updateV_avg}, and hence it yields a valid solution.
		Next, we verify that ${\mathcal{V}}^t\in \bar{\mathcal{F}}_{\mathcal{V}}$. Since the initialization satisfies ${\mathcal{V}}^0 \in \bar{\mathcal{F}}_{\mathcal{V}}$, we proceed by induction. Suppose ${\mathcal{V}}^{t-1} \in  \bar{\mathcal{F}}_{\mathcal{V}}$. For any $i\in [n]$, based on the definition of $\hat{\mathcal{R}}^t$ in \eqref{eq:imputingR} and the fact that $({U}^{t-1})_{i:} = e_p$ for some $p\in[k]$, we have 
		\begin{align*}
			(\hat{\mathcal{R}}^t)_{ij:} = 
			\left\{ \begin{aligned}
				& \mathcal{R}_{ij:}, \quad &\textrm{if } j \in\Omega_i, \\
				& ({\mathcal{V}}^{t-1})_{pj :} , \quad &\textrm{if } j \notin\Omega_i.
			\end{aligned}\right.
		\end{align*}
		Since $\mathcal{R}_{ij:} \in \{0,1\}^l$ and $\sum_{q=1}^l \mathcal{R}_{ijq} = 1$ for all $(i,j) \in \Omega$, and $\mathcal{V}^{t-1} \in \bar{\mathcal{F}}_{\mathcal{V}}$ ensures that $\sum_{q=1}^l \mathcal{V}_{pjq} = 1$, it follows that $\hat{\mathcal{R}}^t \in [0,1]^{n \times m \times l}$ and satisfies $\sum_{q=1}^l (\hat{\mathcal{R}}^t)_{ijq} = 1$ for all $i \in [n]$ and $j \in [m]$. Therefore, ${\mathcal{V}}^t$ given by \eqref{eq:updateV_avg} satisfies ${\mathcal{V}}^t\in\bar{\mathcal{F}}_{\mathcal{V}}$, completing the induction.
		
		Noting that problem~\eqref{pr:ALS-sub-U} is row-wise separable and each row $(U^t)_i$ is a one-hot vector indicating the group in $\mathcal{V}^t$ whose representation best matches the observed entries of the $i$-th row of $\mathcal{R}$, we obtain the closed-form solution given in \eqref{eq:def_assignment_U}.
	\end{proof}
	
	Notably, when the data tensor $\mathcal{R}$ is fully observed, i.e., $\Omega = [n] \times [m]$, Algorithm \ref{alg:alt_min} reduces to the standard alternating minimization method. In the presence of missing entries, however, directly updating $\mathcal{V}$ across all locations based on partially observed data is unreliable, especially in regions with sparse or missing observations, as the objective function in \eqref{pr:mec2} is only defined on the observed index set $\Omega$. To address this, Algorithm \ref{alg:alt_min} performs block-wise updates of $\mathcal{V}$ and $U$, with a key modification that $\mathcal{V}$ is updated using the imputed tensor $\hat{\mathcal{R}}$ defined in \eqref{eq:imputingR}. This imputation-based update enhances robustness to missing data and allows more reliable estimation of $\mathcal{V}$ even in sparsely observed regions.
	
	\subsection{Convergence analysis}
	In this subsection, we show that the iterates generated by TRC in Algorithm~\ref{alg:alt_min} yields a non-increasing sequence of objective values. Moreover, we establish that the generated sequence converges to a blockwise minimum point of problem~\eqref{pr:mec2}, a notion well suited to problems involving mixed discrete and continuous variables or block-wise updates (cf. \cite[Eq.~(4)]{tseng2001convergence} and \cite[Sect.~2]{razaviyayn2013unified}), as formally defined below.
	
	\begin{definition}[Blockwise minimum point of \eqref{pr:mec2}]
		A point $(U^*,\mathcal{V}^*)\in \mathcal{F}_U \times \bar{\mathcal{F}}_\mathcal{V}$ is a blockwise minimum point of \eqref{pr:mec2} 
		if
		\begin{align*}
			& f(U^*,\mathcal{V}^*) \leq f(U^*,\mathcal{V}), \quad \textrm{ for all } \mathcal{V}\in \bar{\mathcal{F}}_{\mathcal{V}},  
			\\
			\mbox{and}\quad & f(U^*,\mathcal{V}^*) \leq f(U,\mathcal{V}^*), \quad \textrm{ for all } U\in \mathcal{F}_U. 
		\end{align*} 
	\end{definition}
	
	We next establish the convergence results for TRC (Algorithm~\ref{alg:alt_min}).
	
	\begin{theorem}
		Let $\{(U^t, \mathcal{V}^t)\}$ be the sequence generated by TRC  in Algorithm~\ref{alg:alt_min}.
		\begin{itemize}[leftmargin=2em]
			\item[(i)] The objective function $\{f(U^t, \mathcal{V}^t)\}$ is monotonically non-increasing, and
			\begin{align}\label{eq:decreasing}
				f(U^{t}, \mathcal{V}^{t}) - f(U^{t+1}, \mathcal{V}^{t+1}) \geq \frac{1}{2} \left\| U^t (\mathcal{V}^{t+1} - \mathcal{V}^{t})\right\|^2, \quad t\geq 0,
			\end{align}
			and hence converge.
			\item[(ii)]  The sequence $\{(U^t, \mathcal{V}^t)\}$ converges to a blockwise minimum point of problem~\eqref{pr:mec2}.
		\end{itemize}
	\end{theorem} 
	\begin{proof}
		(i): According to the update rules in Algorithm~\ref{alg:alt_min}, we have 
		$$
		\begin{aligned}
			&\ f(U^t, \mathcal{V}^{t}) - f(U^{t+1}, \mathcal{V}^{t+1})  \\
			=&\ f(U^t, \mathcal{V}^{t}) - f(U^{t}, \mathcal{V}^{t+1}) +  f(U^{t}, \mathcal{V}^{t+1}) - f(U^{t+1}, \mathcal{V}^{t+1}) \\
			\geq&\ f(U^t, \mathcal{V}^{t}) - f(U^t, \mathcal{V}^{t+1})  \geq \hat{f}_{t+1}(\mathcal{V}^t) - \hat{f}_{t+1}(\mathcal{V}^{t+1}) ,
		\end{aligned}
		$$
		where the first inequality follows from \eqref{pr:ALS-sub-U}, and the second inequality from the fact that 
		$$\begin{aligned}
		\hat{f}_{t+1}(\mathcal{V})=&\ \frac{1}{2}\|P_{\Omega}(\mathcal{R}-U^t\mathcal{V})\|^2+\frac{1}{2}\|P_{\Omega^c} (U^t\mathcal{V}^t-U^t\mathcal{V})\|^2 \\
        = &\ f(U^t,\mathcal{V})+\frac{1}{2}\|P_{\Omega^c} (U^t\mathcal{V}^t-U^t\mathcal{V})\|^2.
		\end{aligned}$$ 
		Moreover, from the Taylor expansion of the quadratic function $\hat{f}_{t+1}(\mathcal{V})$ at $\mathcal{V}^{t+1}$, we can see
		$$
		\begin{aligned}
			& f(U^t, \mathcal{V}^{t}) - f(U^{t+1}, \mathcal{V}^{t+1}) \\
			&\geq  \langle \nabla \hat{f}_{t+1}(\mathcal{V}^{t+1}), \mathcal{V}^{t} - \mathcal{V}^{t+1} \rangle + \frac{1}{2} \left\| U^t (\mathcal{V}^{t+1} - \mathcal{V}^{t})\right\|^2 
			= \frac{1}{2} \left\| U^t (\mathcal{V}^{t+1} - \mathcal{V}^{t})\right\|^2,
		\end{aligned}
		$$
		where the last equality holds due to the optimality of $\mathcal{V}^{t+1}$ in minimizing the unconstrained optimization problem in \eqref{pr:ALS-V}, which implies that $\nabla \hat{f}_{t+1}(\mathcal{V}^{t+1})=0$.
		
		(ii): Since the generated sequence $\{(U^t, \mathcal{V}^t)\}\subseteq \mathcal{F}_U \times \bar{\mathcal{F}}_\mathcal{V}$ is bounded, it admits at least one cluster point. Denote such a cluster point by $(U^*, \mathcal{V}^*)$, and consider a subsequence $\{(U^{s_\nu}, \mathcal{V}^{s_\nu})\}$ converging to $(U^*, \mathcal{V}^*)$ as $\nu\to\infty$.
		Since $\mathcal{F}_U$ is finite and discrete, we must have $U^{s_\nu} = U^*$ for sufficiently large $\nu$. Without loss of generality, assume this holds for all $\nu$.
		
		Define $t_\nu = s_\nu+1$. According to Algorithm \ref{alg:alt_min} and Proposition \ref{prop:sol_V}(i), we can see that $\mathcal{V}^{t_\nu}$ satisfies
		\begin{equation*}
			({\mathcal{V}}^{t_\nu})_{p::}  = \left\{\begin{array}{ll}
				\frac{1}{|J_p^*|}\sum_{i\in J_p^*} ({P}_{\Omega} \left(\mathcal{R} \right) + P_{\Omega^c}   
				\left( {U}^{*}  \mathcal{V}^{s_\nu} \right) )_{i::} & \text{if }\ J_p^*\neq \emptyset, \\[4pt]
				({\mathcal{V}}^{s_{\nu}})_{p::} & \text{if }\ J_p^* =  \emptyset,
			\end{array}\right.  \quad \text{for } p\in[k],
		\end{equation*}
		where $J_p^*=\{i\in[n]\mid (U^*)_{i:}=e_p\}$. 
		Thus, the fact that $\{\mathcal{V}^{s_{\nu}}\}$ converges to $\mathcal{V}^*$ indicates that the sequence $\{\mathcal{V}^{t_{\nu}}\}$ also converges. By denoting the limit of $\{\mathcal{V}^{t_{\nu}}\}$ as $\tilde{\mathcal{V}}^*$, we have
		\begin{equation}\label{eq:updateV_limit}
			(\tilde{\mathcal{V}}^*)_{p::}  = \left\{\begin{array}{ll}
				\frac{1}{|J_p^*|}\sum_{i\in J_p^*} ({P}_{\Omega} \left(\mathcal{R} \right) + P_{\Omega^c}  
				\left( {U}^{*}  \mathcal{V}^{*} \right) )_{i::} & \text{if }\ J_p^*\neq \emptyset, \\[4pt]
				({\mathcal{V}}^{*})_{p::} & \text{if }\ J_p^* =  \emptyset,
			\end{array}\right.  \quad \text{for } p\in[k].
		\end{equation}
		In addition, we know from Algorithm \ref{alg:alt_min} and Proposition \ref{prop:sol_V}(ii) that, $U^{t_\nu}$ is uniquely determined by $\mathcal{V}^{t_\nu}$. This means that 
		the sequence $\{U^{t_\nu}\}$ also converges, with its limit denoted by $\tilde{U}^*$.
		Moreover, by taking $t = s_\nu$ in
		\eqref{eq:decreasing}, we have
		\begin{align*} 
			f(U^{*}, \mathcal{V}^{s_\nu}) - f(U^{t_\nu}, \mathcal{V}^{t_\nu}) \geq \frac{1}{2} \left\| U^* (\mathcal{V}^{t_\nu} - \mathcal{V}^{s_\nu})\right\|^2.
		\end{align*}
		According to part (i), we know that $\{f(U^t, \mathcal{V}^{t})\}_{t\geq 0}$ is non-increasing and converges, which, together with the above inequality, yields
		$$
		f(U^{*}, \mathcal{V}^{*}) = f(\tilde{U}^{*}, \tilde{\mathcal{V}}^{*}) \quad \text{and} \quad 
		\| U^* (\tilde{\mathcal{V}}^{*} - \mathcal{V}^{*})\|^2 = 0. 
		$$
		Then for each $p\in [k]$ such that $J_p^* \neq \emptyset$, we have
		\begin{eqnarray*}
			0 = \| U^* (\tilde{\mathcal{V}}^{*} - \mathcal{V}^{*})\|^2\geq \|(\tilde{\mathcal{V}}^{*})_{p::} - (\mathcal{V}^{*})_{p::}\|^2 = 0.
		\end{eqnarray*}
		And from \eqref{eq:updateV_limit}, we also have
		$\|(\tilde{\mathcal{V}}^{*})_{p::} - (\mathcal{V}^{*})_{p::}\| = 0$ for $J_p^* = \emptyset$. It follows that $\tilde{\mathcal{V}}^{*} = \mathcal{V}^{*}$, and consequently $\tilde{U}^* = U^*$.  
		
		As $\mathcal{V}^{t_\nu}$ is the optimal solution to \eqref{pr:ALS-V} with $t = t_\nu$, we have 
		\begin{equation*}
			\begin{aligned}
				\tilde{\mathcal{V}}^{*} =&\ \underset{\mathcal{V}\in\mathbb{R}^{k\times m\times l}}{\argmin}  
				\left\{	\frac 1 2 \left\|\hat{\mathcal{R}}^{*} - {U}^{*}\mathcal{V} \right\|^2 \right\},
				\quad \text{where} \quad 
				\hat{R}^*:={P}_{\Omega} \left(\mathcal{R} \right) + P_{\Omega^c}  
				\left( {U}^* \mathcal{V}^* \right).
			\end{aligned}
		\end{equation*}
		By the optimality of the above minimization problem, we can see that
		\begin{equation*}
			(U^*)^\top(U^* \tilde{\mathcal{V}}^{*}  - P_{\Omega}(R) - P_{\Omega^c}(U^* \mathcal{V}^{*}) )=0,
		\end{equation*}
		which, along with the fact that $\tilde{\mathcal{V}}^{*} =\mathcal{V}^* $, implies
		\begin{equation*}
			\nabla_{\mathcal{V}} f(U^*, \mathcal{V}^*) = (U^*)^\top(P_{\Omega}(U^* \mathcal{V}^{*} - R)=0.
		\end{equation*}
		Since $f(U^*, \cdot)$ is convex, we obtain 
		\begin{equation}\label{eq:partial_optimal_V}
			f(U^*, \mathcal{V}) \geq f(U^*, \mathcal{V}^*) + \langle \nabla_{\mathcal{V}} f(U^*, \mathcal{V}^*) , \mathcal{V} - \mathcal{V}^*\rangle = f(U^*, \mathcal{V}^*) , \quad \forall\ \mathcal{V} \in \bar{\mathcal{F}}_{\mathcal{V}}. 
		\end{equation}
		From the update rule of $U$, 
		it also holds that $ f(U, \mathcal{V}^{s_\nu}) \geq f(U^{s_\nu}, \mathcal{V}^{s_\nu}) $ for all $U\in\mathcal{F}_U$.
		By the continuity of $f(\cdot,\cdot)$, letting $\nu \to \infty$, we obtain 
		\begin{equation}\label{eq:partial_optimal_U}
			f(U, \mathcal{V}^{*}) \geq f(U^{*}, \mathcal{V}^{*}), \quad \forall \ U\in\mathcal{F}_U.
		\end{equation}
		Combining \eqref{eq:partial_optimal_V} and \eqref{eq:partial_optimal_U}, we conclude that $(U^*, \mathcal{V}^*)$ is a blockwise minimum point of problem~\eqref{pr:mec2}. 
		
		Finally, we have shown that if $(U^{s_{\nu}}, \mathcal{V}^{s_{\nu}}) \to (U^{*}, \mathcal{V}^{*})$, then 
		$(U^{s_{\nu}+1}, \mathcal{V}^{s_{\nu}+1}) \to (U^{*}, \mathcal{V}^{*})$. By induction, it follows that  
		$$(U^{s_{\nu}+j}, \mathcal{V}^{s_{\nu}+j}) \to (U^{*}, \mathcal{V}^{*}) \quad \text{for all } j=1,2,3,\ldots .$$
		Consequently, the entire sequence $\{(U^t, \mathcal{V}^t)\}$ converges to the parial optimal point $(U^{*}, \mathcal{V}^{*})$. 
	\end{proof}
	
	\section{Enhanced LCMC: Addressing scalability and data imbalance}
	\label{sect:enhanced}
	In this section, we present two enhanced strategies to address the computational challenges that arise in large-scale and imbalanced datasets. 
	
	\subsection{Split-merge-refine approach for large-scale data}\label{sect:eBCMC} 
	A classical way to tackle large-scale problems is the divide-and-conquer strategy (e.g., \cite{knuth1998art}), which partitions a problem into smaller, more manageable subproblems to improve computational efficiency. Inspired by this idea, we propose a split-merge-refine strategy to address the challenges of latent structural categorical matrix completion posed by large-scale data, with details given as follows.
	
	\textbf{Split.} We first randomly partition the data and its observed index set into $S$ parts, where $S$ is a user-defined parameter determined by the available computational resources so that each subproblem remains manageable. This yields $S$ independent subproblems, each associated with a subdataset $(\mathcal{R}^{(s)}, \Omega^{(s)})$ for $s \in[S]$. Each subproblem is then solved independently using LCMC to obtain a coarse estimate $\mathcal{V}^{(s)}$ of the latent tensor on that subset. 

	\textbf{Merge.}
	The latent dimension $k$ is initially overestimated after concatenating all $\mathcal{V}^s$. In the local merging step, 
	we first perform a pairwise merge test to determine whether two categories should be combined, using the internal Minimum Error Correction (MEC) score of each category as the evaluation criterion. Here we define the internal MEC for category $v$ as:
	\begin{eqnarray}\label{eq:withinMEC}
		\operatorname{MEC}_v = \frac{1}{2} \sum_{i\in I(v)}  \|{P}_{\Omega_i}(\mathcal{R}_{i::} -v)\|^2,
	\end{eqnarray}
	where the membership index set $I(\cdot)$ is
	\begin{eqnarray}\label{def:member_index}
		I(v) = \{i\in[n]\ |\ \textrm{data } \mathcal{R}_{i::} \textrm{  is assigned to } v \}. 
	\end{eqnarray} 
	Note that this membership index set can be directly obtained from the assignment matrix $U$. If the corresponding $U$ is unavailable, it can be computed following Proposition \ref{prop:sol_V}(ii). The merge procedure is summarized in Algorithm \ref{alg:merge}. 
	
    %\vspace*{-0.5cm}
	\begin{algorithm}[H]
		\caption{Local merge: (MergeFlag, $w$, $\mathcal{V}$) $\gets$ \textbf{MergeTest}($\mathcal{V},U$)}\label{alg:merge}
		\begin{algorithmic}[1]
			\REQUIRE Representative tensor $\mathcal{V}$ and the corresponding categorical encoding  $U$. 
			\renewcommand{\algorithmicrequire}{\textbf{Initialization:}}
			\REQUIRE Denote the initial latent dimension as $k$. Set MergeFlag $= 0$.
			\STATE Compute $\operatorname{MEC}_{\mathcal{V}_{1::}}$ by \eqref{eq:withinMEC}.
			\FOR{$j=2, \ldots, k$}
			\STATE Compute $\operatorname{MEC}_{\mathcal{V}_{j::}}$, and  $\operatorname{MEC}_{w}$ for $w$ with  $I(w)=I(\mathcal{V}_{1::})\cup I(\mathcal{V}_{j::})$. 
			\IF {$\operatorname{MEC}_{w}<\min\{\operatorname{MEC}_{\mathcal{V}_{1::}}, \operatorname{MEC}_{\mathcal{V}_{j::}}\}$}
			\STATE Remove $\mathcal{V}_{1::}$ and $\mathcal{V}_{j::}$ from $\mathcal{V}$; set MergeFlag $= 1$; break.
			\ENDIF
			\ENDFOR
			\ENSURE MergeFlag, $w$, $\mathcal{V}$. 
		\end{algorithmic}  
	\end{algorithm}
	%\vspace*{-0.5cm}

	\textbf{Refine.}  
	We then apply a global refinement step based on overall improvement in SSE. Specifically, for $\mathcal{V}\in\mathbb{R}^{k\times m\times l}$, we define the SSE at $\mathcal{V}$ as
	\begin{eqnarray}\label{eq:SSE_all}
		{\rm SSE}(\mathcal{V}) = \sum_{p=1}^k \sum_{i \in I(\mathcal{V}_{p::})}\ \|P_{\Omega_i}(\mathcal{R}_{i::} - \mathcal{V}_{p::})\|^2,
	\end{eqnarray}
	where the membership index set $I(\cdot)$ is defined in \eqref{def:member_index}. 
	We start with the solution obtained from the merge step and iteratively assess whether the change in SSE is acceptable when the category with the lowest frequency is removed. 
	Overall, this step aims to fine-tune the result by eliminating less significant categories while maintaining a good fit to the data.
	
	Given all the procedures presented, the enhanced LCMC is given in Algorithm \ref{alg:random_split}.
	\begin{algorithm}[!h]
		\caption{Enhanced LCMC for large-scale dataset: $(U,\mathcal{V})\gets$\textbf{eLCMC}($\mathcal{R}, \Omega$)}\label{alg:random_split}
		\begin{algorithmic}[1]
			\REQUIRE Data tensor $\mathcal{R}\in \{0,1\}^{n\times m\times l}$ and observed index set $\Omega\subseteq [n]\times [m]$. 
			\renewcommand{\algorithmicrequire}{\textbf{Initialization:}}
			\REQUIRE Choose $\epsilon>0$ (in our experiments, we set $\epsilon=10^{-2}$), $\textrm{SSE}_{\textrm{old}}=mn$ and the number of subproblems as $S$.
			\STATE Randomly split $\mathcal{R}$ and $\Omega$ by rows into $S$ parts $\{\mathcal{R}^s,\Omega^s\}_{s\in[S]}$.
			\FOR{$s=1,\ldots,S$}
			\STATE $(U^{(s)},\mathcal{V}^{(s)}) \gets$ \textbf{LCMC}($\mathcal{R}^{(s)}, \Omega^{(s)}$) by Algorithm \ref{alg:BCMC}.
			\ENDFOR
			\STATE Concatenate $\mathcal{V}\! = \![\mathcal{V}^{(1)}; \ldots; \mathcal{V}^{(s)}]$ and $U\!=\![U^{(1)};\ldots; U^{(s)}]$. 
			Set $\hat{\mathcal{V}}=\emptyset$.
			\WHILE{number of categories in $\mathcal{V}$ exceeds $1$}%
			\STATE (MergeFlag, $v$, $\mathcal{V}$) $\gets$ \textbf{MergeTest}($\mathcal{V}, U$) by Algorithm \ref{alg:merge}.
			\IF{MergeFlag} 
			\STATE Add $v$ to $\hat{\mathcal{V}}$. 
			\ELSE
			\STATE Add $\mathcal{V}_{1::}$ to $\hat{\mathcal{V}}$. Remove $\mathcal{V}_{1::}$ from $\mathcal{V}$ and update $U$ accordingly.
			\ENDIF
			\ENDWHILE
			\STATE Concatenate $\hat{\mathcal{V}}$ to $\mathcal{V}$. 
			Let $k$ be the number of categories in $\mathcal{V}$.
			\FOR{$i=1\ldots,k$}
			\STATE Let $\tilde{\mathcal{V}}$ be $\mathcal{V}$ without its least frequent category. Compute
			$\textrm{SSE}(\tilde{\mathcal{V}})$ by \eqref{eq:SSE_all}.  \\
			\textbf{if} {${\textrm{SSE}(\tilde{\mathcal{V}})}/{\textrm{SSE}_{\textrm{old}}} \geq 1+\epsilon$} \textbf{then} break 
			\textbf{end if}
			\STATE Set ${\mathcal{V}} = \tilde{\mathcal{V}}$ and  $\textrm{SSE}_{\textrm{old}}=\textrm{SSE}( \tilde{\mathcal{V}})$.
			\ENDFOR
			\ENSURE Representative tensor $\mathcal{V}$ and the corresponding assignment matrix ${U}$. 
		\end{algorithmic}  
	\end{algorithm}
    
	\subsection{Adaptive data reduction for imbalanced dataset}\label{sect:adaTenRec}
	Since the objective of \eqref{pr:mec2} is to minimize the restricted distances between data points and their assigned representatives, categories with substantially more data points tend to dominate the objective. This biases the solution toward majority patterns, making it difficult to detect rare components. 
	
	To address this issue, we propose an adaptive data reduction strategy within the TRC framework to reduce the influence of the dominant patterns and enhance the detection of rare ones. Specifically, we first solve the completion problem to identify the dominant category. Data points assigned to this dominant category are then removed, and TRC is rerun on the reduced dataset. This process is repeated iteratively to progressively recover distinct categories.
	A key challenge is reliably identifying which samples are associated with the dominant category. Although the matrix $U$ provides the assignment information, it may be unreliable in imbalanced settings. To improve robustness, we introduce a criterion based on the individual MEC, defined for a data point $\mathcal{R}_{i::}$ and representative $v$ as
	\begin{equation}\label{eq:singleMEC}
		\operatorname{MEC}_v^i = \frac{1}{2}\|{P}_{\Omega_i}(\mathcal{R}_{i::} - v)\|^2.
	\end{equation}
	The complete adaptive data reduction procedure is outlined in Algorithm \ref{alg:adaptive}.
	
	%\vspace{-0.3cm}
	\begin{algorithm}[H]
		\caption{Tensor completion with adaptive data reduction : $(\hat{U},\hat{\mathcal{V}}) \gets$\textbf{aTRC}($\mathcal{R}, \Omega, k$)}\label{alg:adaptive}
		\begin{algorithmic}[1]
			\REQUIRE  Data tensor $\mathcal{R}$, observed index set $\Omega$, and latent dimension $k$.
			\renewcommand{\algorithmicrequire}{\textbf{Initialization:}}
			\REQUIRE Set initial $\hat{\mathcal{V}} = \emptyset$. 
			\WHILE{$\|{\mathcal{R}}\|\geq 0$ and $k>1$}
			\STATE Compute tensor completion by Algorithm \ref{alg:alt_min}: $(U,\mathcal{V})\gets \textbf{TRC}(\mathcal{R}, \Omega, k)$.
			\STATE Identify the most dominant category $v$ in $\mathcal{V}$.
			\STATE Remove $\mathcal{R}_{i::}$ and $\Omega_{i}$ for all $i \in I(v)$ with $\operatorname{MEC}_v^i$ in \eqref{eq:singleMEC} below average.
			\STATE Add slice $v$ to $\hat{\mathcal{V}}$, and decrement $k \gets k-1$.
			\ENDWHILE
			\ENSURE Representative tensor $\hat{\mathcal{V}}$ and the corresponding assignment matrix $\hat{U}$.
		\end{algorithmic}  
	\end{algorithm}
	
	\begin{remark}
		In Algorithm \ref{alg:adaptive}, TRC may be called multiple times, making it the most computationally expensive component. However, it is only re-executed after each round of data reduction, and the dataset size shrinks with each iteration, leading to a substantial decrease in cost over time. As a result, the overall computational burden remains manageable. Furthermore, most of the other computations, such as evaluating MEC and SSE, are already integrated into LCMC (see Section \ref{sec: upper_level}) and incur minimal additional overhead.
	\end{remark}

	\section{Application to viral quasispecies analysis}\label{sec: viral}
	As an application of our proposed latent structural categorical matrix completion framework, we consider the problem of viral quasispecies (VQS) reconstruction. In this setting, sequencing reads correspond to partially observed categorical entries, with each position representing one of a finite set of nucleotides. The objective is to reconstruct the underlying viral strains that best explain these incomplete observations. This problem fits naturally into our framework and highlights its practical relevance.
	
	%\vspace{-0.5cm}
	\begin{figure}[H]
		\centering
		\includegraphics[width=0.8\textwidth]{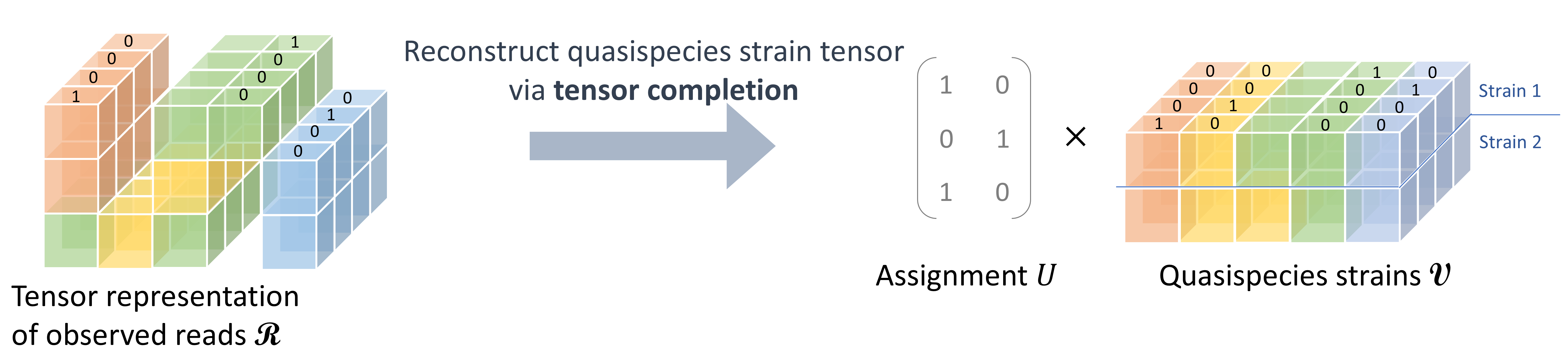}
		\vspace{-0.2cm}
		\caption{
			Example of tensor representation of VQS with $n=3$, $m=5$ and $k=2$. }
		\label{fig:tensor_factorization}
	\end{figure}
	%\vspace{-0.8cm}

	\subsection{Tensor representation of quasispecies analysis}
	Let $n$ be the number of sequencing reads, and $m$ denote the length of each viral strain, i.e., the number of nucleotide positions (also called sites or bases). Each base takes one of four possible nucleotides ($l=4$), \{A, C, G, U/T\}, which we represent using one-hot encoding as $\left\{(1,0,0,0),(0,1,0,0),(0,0,1,0),(0,0,0,1)\right\}$. The sequencing reads are encoded as a tensor $ \mathcal{R}=P_{\Omega}(\mathcal{R})\in \{0,1\}^{n\times m\times 4}$, where for each $(i,j)\in \Omega$, $\mathcal{R}_{ij:}$ indicates the nucleotide observed at site $j$ in read $i$. Here $\Omega\subseteq [n]\times [m]$ denotes the observed index set of the sequencing reads.
	Suppose the quasispecies contains $k$ distinct strains, represented by a tensor $\mathcal{V}\in \{0,1\}^{k \times m\times 4}$. 	
	Our goal is to estimate the number of strains $k$ and the quasispecies tensor $\mathcal{V}\in \{0,1\}^{k\times m\times 4}$, based on the partially observed tensor ${P}_{\Omega}(\mathcal{R})$, see Figure \ref{fig:tensor_factorization}.
			
			\subsection{Numerical results}
			\label{sec: experiment}
			We conduct numerical experiments to assess the effectiveness of our proposed LCMC for quasispecies analysis on short-length reads, long-length reads, and large-scale and imbalanced data. Both LCMC and the benchmark method TenSQR \cite{ahn2018viral} are implemented in Python, ensuring a fair comparison in terms of computational efficiency. TenSQR is a well-established optimization-based solver for quasispecies analysis and serves as a strong benchmark. As demonstrated in \cite{ahn2018viral}, it outperforms other publicly available software tools for quasispecies analysis, such as aBayesQR \cite{ahn2018abayesqr}, ViQuaS \cite{jayasundara2015viquas}, PredictHaplo \cite{prabhakaran2013hiv}, and ShoRAH \cite{zagordi2011shorah}. This enables us to showcase the advantages of LCMC in terms of efficiency, accuracy, robustness and scalability. 
			
			All experiments were run on a desktop with a 13th Gen Intel(R) Core(TM) i7-13700 2.10 GHz processor and 32GB of RAM.
			Sections \ref{sect:synthetic} and \ref{sect:real} present the results of {LCMC} on short-length and long-length reads data, respectively.
			Section \ref{sect:numerical_enhanced} examines the performance of the enhanced techniques introduced in Section~\ref{sect:enhanced}  on large-scale and imbalanced data.
			
			\subsubsection{Experiments on short-length read data}
			\label{sect:synthetic}
			To systematically evaluate the performance of LCMC, we generate synthetic short-length read datasets with known ground truth, which allow us to assess the solver's accuracy, efficiency and robustness under controlled conditions. We vary key parameters, such as sequencing error rate, strain diversity, and the number of strains, to analyze how LCMC performs across different problem settings.
			
			We generate the data by emulating high-throughput sequencing (HTS) of quasispecies samples using the short-read simulator ART\footnote{\url{https://www.niehs.nih.gov/research/resources/software/biostatistics/art}} \cite{huang2012art}, which replicates the sequencing process based on empirical error models or user-defined quality profiles. Specifically, we first randomly generate a reference genome of length 1300 bp, a typical length for the HIV-1 pol region and also used in TenSQR. Viral strains are then created by introducing independent mutations at uniformly random locations in the reference genome, with the diversity level measured by the average Hamming distance between strains. Then, we generate $n$ paired-end reads of length $2 \times 250$ bp, uniformly sampled from a mixture of viral strains, while incorporating a given sequencing error rate. These reads are then aligned to the reference genome using the BWA-MEM\footnote{\url{https://github.com/lh3/bwa}} \cite{li2013aligning}, a tool for aligning short reads to a reference genome and filtering out low-quality reads. Since the strain reconstruction procedure only utilizes the heterozygous sites (i.e., the sites containing allele variants), rather than the full-length reads, we further preprocess the aligned data by removing non-informative sites. The final processed read data serves as the input for our quasispecies analysis.
			
			We test the performance of LCMC and TenSQR under various parameter settings, including different diversity levels, sequencing error rates, and strain mixtures. The diversity level (div) in our experiments ranges from 1\% to 5\%, and the sequencing error rate (err) is set to $0.7\%$ or $0.2\%$. We consider mixtures of five and ten strains with frequencies (0.5, 0.3, 0.15, 0.04, 0.01) and (0.36, 0.24, 0.16, 0.08, 0.055, 0.04, 0.03, 0.02, 0.01, 0.005), respectively. The number of generated reads is set as $n=6500$, with the actual number after filtering typically ranging from 6400 to 6500. 
			
			The numerical result is presented in Tables \ref{tab:k_5} and \ref{tab:k_10}. LCMC demonstrates comparable performance to TenSQR, a strong benchmark approach for quasispecies analysis using short-length reads, in terms of both running time and solution quality. LCMC shows similar performance across most scenarios and outperforms TenSQR in cases with low diversity. 
			Notably, TenSQR exhibits instability in quasispecies size estimation in certain instances, where it occasionally fails to terminate due to uncontrolled growth in inferred strain numbers. In contrast, LCMC maintains robust performance and reliably produces well-structured solutions under the same conditions.
			
			%\vspace{-0.2cm}
			\begin{table}[H]
				\begin{center}
					\caption{Results on synthetic data with the truth quasispecies size $\bar{k}=5$ (average over 50 instances).
						``$|\Omega|/mn (\%) $" denotes the observed data percentage. 
						``E(N)/F" indicates the number of exact ($k=\bar{k}$), near-exact ($\bar{k}\pm 1$), and failed cases. ``Time" reports the CPU time in seconds.
					} 
					\vspace{-0.5em} % Part II
					\label{tab:k_5}
					\footnotesize
					\begin{tabular}{c|c|c|c|c|c|c}
						\hline 
						\rule{0pt}{1.2em} \multirow{2}{*}{ \shortstack{err\\$10^{-3}$}} & \multirow{2}{*}{\shortstack{div\\ \%}} & \multirow{2}{*}{ \shortstack{ ${|\Omega|}/{mn}$ \\ \%}} & \multicolumn{2}{|c|}{TenSQR}  & \multicolumn{2}{|c}{LCMC}  \\ \cline{4-7} 
						\rule{0pt}{1.2em}   & & & E(N)/F &Time (s)  & E(N)/F &Time (s)  \rule[-0.6em]{0pt}{1em} 
						\\ \hline
						\rule{0pt}{1.2em}    
						\multirow{5}{*}{$2$} 
						&1   & 48.9 &2(20)/4 &13 & 30(37)/- & 17   \\
						\rule{0pt}{1.2em}  &2  & 48.4 &2(24)/- & 11 & 36(41)/- & 15  \\
						\rule{0pt}{1.2em}  &3  & 48.5 &36(40)/- & 10 & 37(44)/- & 14  \\
						\rule{0pt}{1.2em}  &4  & 49.1 & 39(43)/- & 10 & 42(45)/- & 13  \\ 
						\rule{0pt}{1.2em}  &5  & 48.1 & 42(45)/- & 11 & 42(46)/- & 15  \\ 
						\hline
						\rule{0pt}{1.2em}    \multirow{5}{*}{$7$} 
						&1  	& 48.6 & 8(24)/4 &13 &20(27)/- & 19   \\ 
						\rule{0pt}{1.2em}  &2  & 48.7 & 32(39)/- & 10 &35(40)/- & 14  \\
						\rule{0pt}{1.2em}  &3  & 49.0 & 34(36)/- & 10 & 36(40)/- & 13 \\
						\rule{0pt}{1.2em}  &4  & 48.1 & 40(41)/- &10 & 40(43)/- & 13   \\
						\rule{0pt}{1.2em}  &5  & 49.2 & 40(43)/- &13 & 41(45)/- & 15  \\ 
						\hline 
					\end{tabular}		
				\end{center}
			\end{table}
			\vspace{-0.5cm}
			\begin{table}[H]
				\begin{center}
					\caption{Results on synthetic data with the truth quasispecies size $\bar{k}=10$ (average over 50 instances).
					} 
					\vspace{-0.5em} % Part II
					\label{tab:k_10}
					\footnotesize
					\begin{tabular}{c|c|c|c|c|c|c}
						\hline 
						\rule{0pt}{1.2em} \multirow{2}{*}{ \shortstack{err\\$10^{-3}$}} & \multirow{2}{*}{\shortstack{div\\ \%}} & \multirow{2}{*}{ \shortstack{ ${|\Omega|}/{mn}$ \\ \%}}  & \multicolumn{2}{|c|}{TenSQR}  & \multicolumn{2}{|c}{LCMC}   \\ \cline{4-7} 
						\rule{0pt}{1.2em}   & & & E(N)/F &Time (s)  & E(N)/F &Time (s)  \rule[-0.6em]{0pt}{1em} 
						\\ \hline
						%$2\times 10^{-3}$ &$9$ &$6500$ & 5 &48.8 & 500046 & 5/8/8 & 2.10 / 4.97 \\
						\rule{0pt}{1.2em}    
						\multirow{5}{*}{$2$} 
						&1   & 48.9   &5(23)/1 &112 & 6(23)/- & 87 \\
						\rule{0pt}{1.2em}  &2  & 48.4 &13(27)/- &89  & 11(24)/- & 103 \\
						\rule{0pt}{1.2em}  &3  & 48.5 &12(28)/- &96 & 11(28)/- & 89\\
						\rule{0pt}{1.2em}  &4  & 49.1  &21(27)/- &105 & 21(29)/- & 107\\ 
						\rule{0pt}{1.2em}  &5  & 48.1  &27(30)/- &113 & 27(29)/- & 121\\    
						\hline
						\rule{0pt}{1.2em}    \multirow{5}{*}{$7$} 
						&1  	& 48.6  &10(23)/- &108 & 11(25)/- & 103 \\ 
						\rule{0pt}{1.2em}  &2  & 48.7  &11(23)/- &121 & 13(23)/- & 118\\
						\rule{0pt}{1.2em}  &3  & 49.0  &15(27)/- &95 & 13(25)/- & 103\\
						\rule{0pt}{1.2em}  &4  & 48.1  &19(26)/- &102 & 21(25)/- & 113 \\
						\rule{0pt}{1.2em}  &5  & 49.2 &25(29)/- &125 & 26(31)/- & 130\\ 
						\hline 
				\end{tabular}		\end{center}
			\end{table}
			
			\subsubsection{Experiments on long-length read data}
			%	\subsection{Experiments on real data}
			\label{sect:real} 
			We present the numerical result on real-world datasets consisting of long-length reads/contigs, downloaded from the NCBI SARS-CoV-2 Data Hub\footnote{\url{https://www.ncbi.nlm.nih.gov/}}. 
			To prepare the data for the experiments, we first downloaded isolates from the NCBI SARS-CoV-2 Data Hub, using the `Pangolin' classification as the reference label for each isolate. The raw sequences were then aligned to the Wuhan reference genome (NC\_045512) using the MAFFT alignment tool. Following alignment, we extracted a matrix consisting of only the heterozygous sites across all reads, which served as the input data for the subsequent quasispecies reconstruction process. 
			We downloaded 1021 isolates, categorized into five groups based on Pangolin labels: \emph{A, AY.100, B.1.1.7, XBB.1.5, JD.1.1}. After alignment, the sequences had a length of 29909, with 29266 remaining heterozygous sites. 
     
            %\vspace{-0.5cm}
			\begin{figure}[H]
				\centering
				\includegraphics[width=\textwidth]{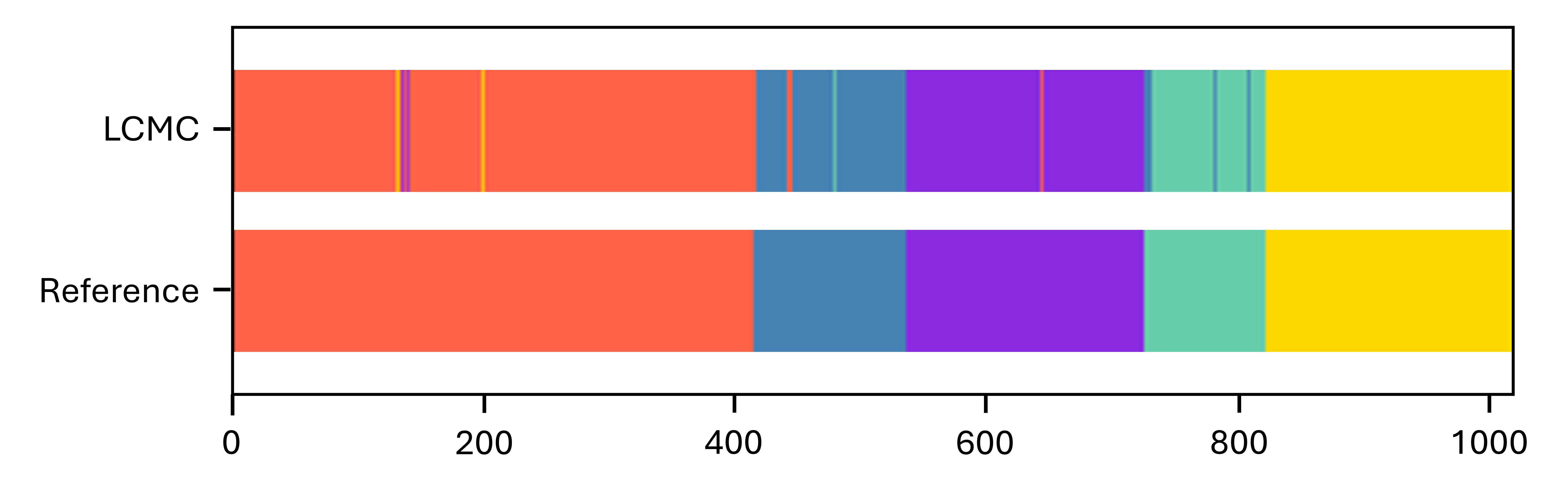}
				\caption{Comparison of estimated labels from LCMC with the reference labels, where different colors represent different strains.}
			\label{fig:real_test_1}
		      \end{figure}
		      %\vspace{-0.5cm}    
			
			LCMC completed the analysis in 351.99 seconds, estimating the quasispecies size as $k=5$ (refining by removing strains with membership size fewer than 5\%). In contrast, TenSQR failed to terminate within 6 hours, with the estimated value of $k$ continuously increasing and eventually reaching 192. A comparison between the reference labels and the LCMC solution is shown in Figure~\ref{fig:real_test_1}.

        %\vspace{-0.5cm}
		\begin{figure}[H]
			\centering
			\includegraphics[width=\textwidth]{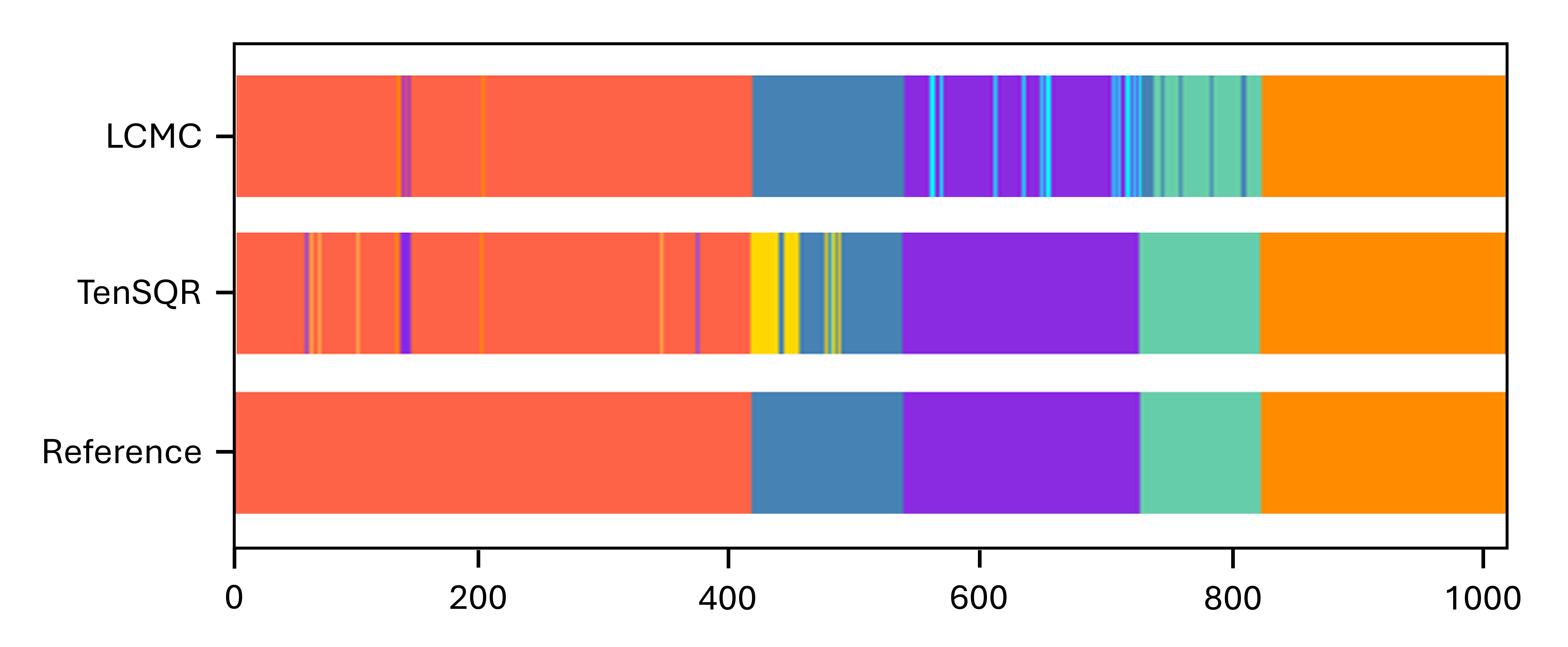}
			\caption{Comparison of estimated labels from LCMC and TenSQR with the reference labels.}
			\label{fig:real_test_1_cliped}
		\end{figure}
		%\vspace{-0.5cm}
        
		To enable TenSQR to produce a reasonable result, we clipped the read length to 8500 and reran the experiment. Under this condition, LCMC completed the analysis in 323.22 seconds, estimating $k=6$ with 63 mislabeled reads (after refinement). TenSQR, in comparison, ran for 819.92 seconds, estimated $k=8$, and resulted in 69 mislabels. A comparison of the solutions obtained from LCMC, TenSQR, and the reference labels is presented in Figure~\ref{fig:real_test_1_cliped}.
		Overall, LCMC consistently exhibited greater stability and efficiency compared to TenSQR. Even under conditions adjusted to favor TenSQR’s convergence, LCMC achieved faster runtimes, fewer mislabeled reads, and more accurate estimates of the number of clusters. These findings underscore LCMC’s robustness and effectiveness in quasispecies analysis.
		
		\subsubsection{Experiments on large-scale and imbalanced dataset}
		\label{sect:numerical_enhanced}
		We evaluate the enhanced techniques introduced in Section~\ref{sect:enhanced} of LCMC on large-scale and imbalanced datasets, which were downloaded and preprocessed as described in Section~\ref{sect:real}. TenSQR is excluded due to its inability to process such large-scale datasets efficiently.      
        As part of this evaluation, we further investigate the split-merge-refine approach under different partition settings. Specifically, we consider two choices for the number of subproblems in Algorithm~\ref{alg:random_split}, namely $S=5$ and $S=8$, and for each choice repeat the experiment 60 times with different random seeds to generate different partitions, thereby assessing the robustness of the splitting step.
        
        Performance is evaluated using assignment accuracy together with group-size-weighted precision, recall, and F1 score. Specifically, let $N$, $K$, and $n_i$ denote the total number of samples, the number of strains, and the number of samples associated with the $i$-th strain, respectively, and define the evaluation metrics as:
        \begin{align*}
        \mathrm{Accuracy} &= \frac{1}{N}\sum_{j=1}^{N}\mathbf{1}(y_j=\hat{y}_j), \quad \;
        \mathrm{Precision} = \sum_{i=1}^{K}\frac{n_i}{N} \frac{\operatorname{TP}_i}{\operatorname{TP}_i+\operatorname{FP}_i}, \\
        \mathrm{Recall} &= \sum_{i=1}^{K}\frac{n_i}{N} \frac{\operatorname{TP}_i}{\operatorname{TP}_i+\operatorname{FN}_i}, \quad
        \mathrm{F1\ score} = \sum_{i=1}^{K}\frac{n_i}{N} \frac{2\operatorname{TP}_i}{2\operatorname{TP}_i + \operatorname{FP}_i + \operatorname{FN}_i}.
        \end{align*}
        Here, $y_j$ and $\hat{y}_j$ denote the true and predicted labels of the $j$-th sample, respectively, while $\operatorname{TP}_i$, $\operatorname{FP}_i$, and $\operatorname{FN}_i$ denote the numbers of correctly assigned, falsely assigned, and missed samples for the $i$-th true group, respectively.
		
		\textbf{Balanced dataset.} 
		The test dataset comprises five strain categories—\emph{AY.4}, \emph{BA.2}, \emph{KP.3.1.1}, \emph{LB.1.17}, and \emph{XEC}—with corresponding read counts of 2000, 2000, 2000, 2000, and 1936, respectively. To evaluate the effectiveness of the enhanced LCMC framework, we applied the method described in Section~\ref{sect:enhanced}. 
        Fig.~\ref{fig:numerical_enhanced_balanced} shows the performance of LCMC on this balanced dataset under different splitting settings. All metrics achieve consistently good and comparable mean values, indicating that the proposed framework achieves reliable performance while remaining robust to the choice of reasonable $S$. Precision remains consistently high with low variance, while Recall and F1 score exhibit larger variability, particularly for $S=8$. 
        This is likely because a larger $S$ produces smaller subproblems and requires more merge operations, making the final solution more sensitive to random partitioning and variability in intermediate steps. Therefore, in practice, $S$ is typically chosen as the smallest value that is computationally feasible under the available resources.

        %\vspace{-0.5cm}
		\begin{figure}[H]
			\begin{center}
            \includegraphics[width=0.75\textwidth]{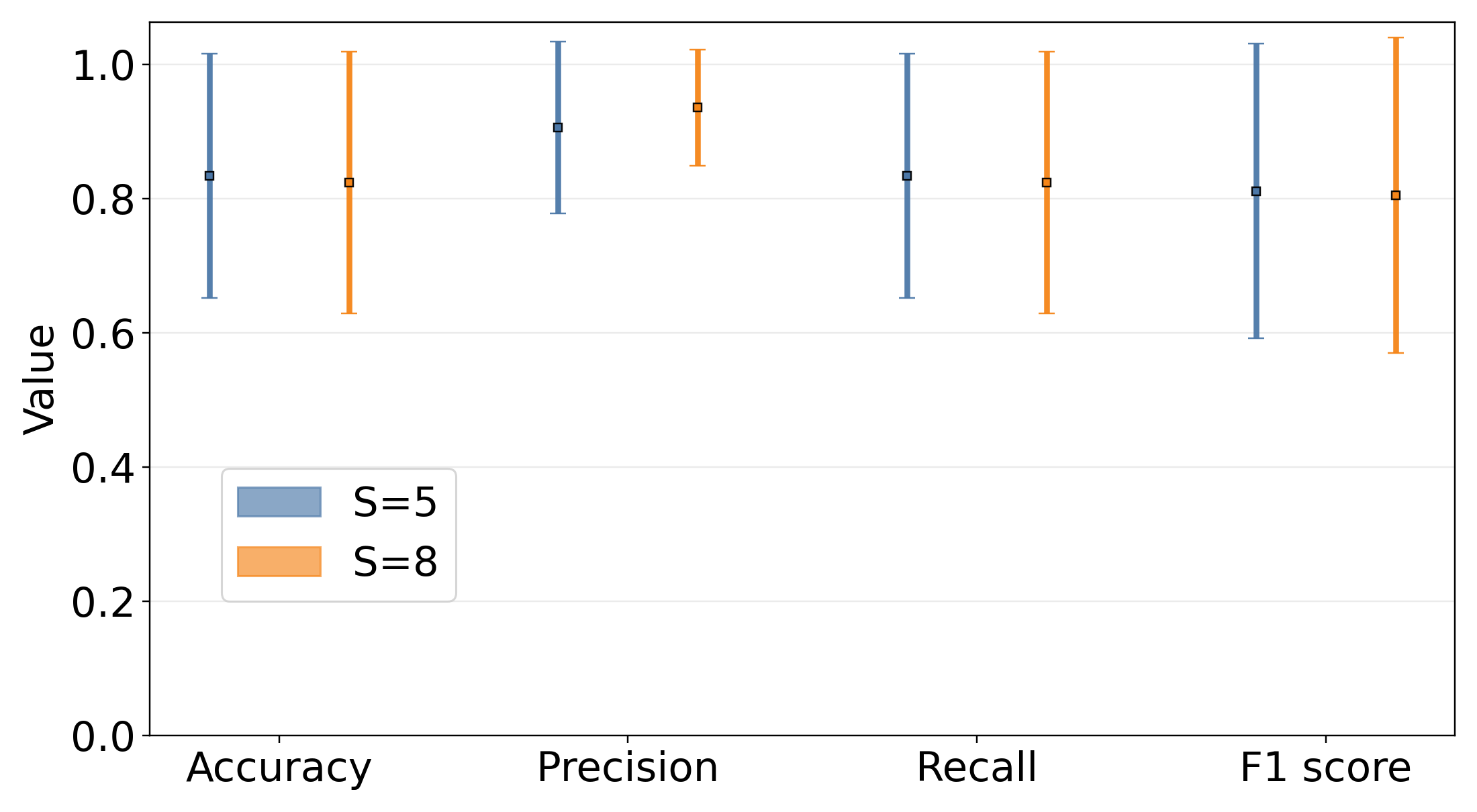}
                \caption{LCMC performance on the large-scale balanced dataset for different choices of the number of subproblems. For each metric, markers denote the mean over 60 independent random partitions, and error bars represent one standard deviation.}
				\label{fig:numerical_enhanced_balanced}
			\end{center}
		\end{figure}
		%\vspace{-0.5cm}
		
		\textbf{Imbalanced dataset.}
		The dataset comprises seven strain categories—\emph{BA.2}, \emph{AY.4}, \emph{KP.3.1.1}, \emph{XBB}, \emph{XEC}, \emph{JN.1}, and \emph{LB.1.17}—with the following number of strains: 2000, 1500, 1000, 1000, 1000, 300, and 100, respectively. 
        Fig.~\ref{fig:numerical_enhanced_imbalanced} presents the results on this imbalanced dataset.
        Compared to Fig.~\ref{fig:numerical_enhanced_balanced}, the mean performance remains comparable or slightly improved across all metrics, while the variability across runs is generally reduced.
        This suggests that the proposed method maintains stable and reliable performance even under data imbalance. Moreover, the results for $S=5$ and $S=8$ remain highly comparable, consistent with the observations in the balanced case.
	
		%\vspace{-0.5cm}
		\begin{figure}[H]
			\begin{center}
            \includegraphics[width=0.75\textwidth]{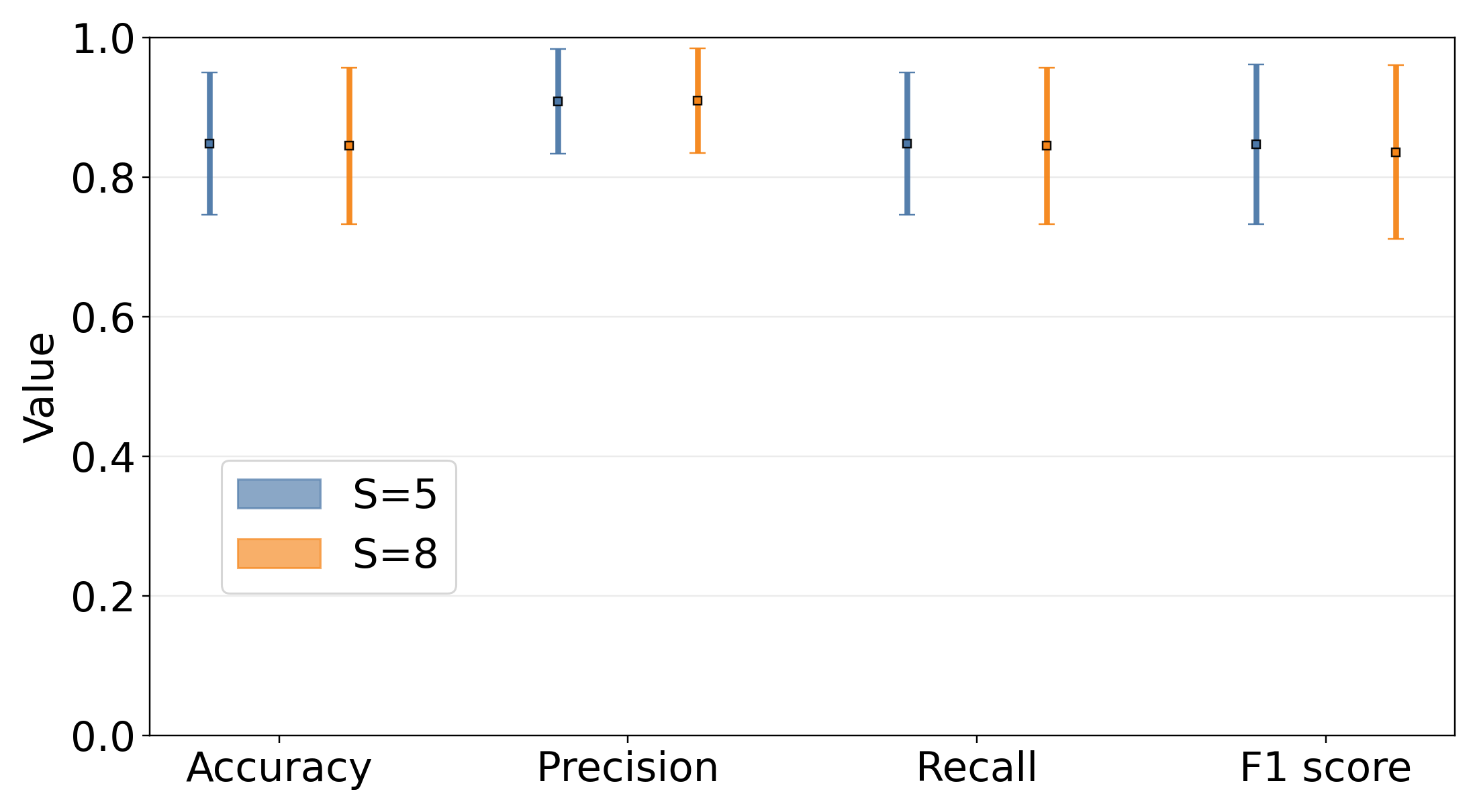}
                \caption{LCMC performance on the large-scale imbalanced dataset for different choices of the number of subproblems. For each metric, markers denote the mean over 60 independent random partitions, and error bars represent one standard deviation.}
				\label{fig:numerical_enhanced_imbalanced}
			\end{center}
		\end{figure}
		%\vspace{-0.5cm}
		
		\section{Conclusion}
		\label{sec: conclusion}
		We introduce LCMC, a double-loop optimization framework for categorical matrix completion with low latent dimension, which effectively addresses the challenges posed by discrete, non-ordinal data.
		Our approach incorporates a tensor-based representation to preserve categorical structure, along with split-merge-refine and data reduction techniques to ensure scalability and robustness in the presence of partial observations, unknown latent dimension, and imbalanced large-scale data. 
		The outer loop estimates the latent dimension via binary search using feedback from the inner loop, while the inner loop solves the tensor completion problem using a tailored factorization algorithm with a theoretical convergence analysis.	
		We validated the effectiveness of LCMC on both synthetic and real datasets, with a particular focus on the viral quasispecies reconstruction problem. Numerical experiments demonstrate that LCMC achieves better accuracy with improved efficiency compared to existing methods. 
		
		Our framework opens up new possibilities for structured learning on categorical data. Future work includes incorporating probabilistic modeling into the LCMC framework to better capture uncertainty in latent assignments and account for observational noise, particularly in biological or clinical applications. To further improve scalability, it is also important to explore online or distributed versions of LCMC that enable real-time analysis in huge-scale or streaming environments.

\bibliographystyle{plain}      % mathematics and physical sciences
\bibliography{ref_LCMC}   % name your BibTeX data base

%% Non-BibTeX users please use
%\begin{thebibliography}{}
%%
%% and use \bibitem to create references. Consult the Instructions
%% for authors for reference list style.
%\end{thebibliography}

\end{document}